\RequirePackage{fix-cm}
\documentclass[smallcondensed,numbook]{svjour3}       
\smartqed  
\usepackage{graphicx}
 \usepackage{mathptmx}      
\usepackage{amsmath}        
\usepackage{amssymb}        
\usepackage{multirow}       
\usepackage[mathscr]{euscript}     
\usepackage[justification=centering]{caption}     
\usepackage{type1cm,type1ec}     
\usepackage[  
pdfpagelabels,hypertexnames=true,
plainpages=false,
naturalnames=false]{hyperref}
\usepackage{color}
\definecolor{darkblue}{rgb}{0,0.1,0.5}
\hypersetup{colorlinks,
linkcolor=darkblue,
anchorcolor=darkblue,
citecolor=darkblue}

\DeclareMathAlphabet{\mathcal}{OMS}{cmsy}{m}{n}    
\newcommand{\fsp}[1]{\mathcal{#1}}       
\newcommand{\Qfs}{\fsp{Q}}      
\newcommand{\solver}{\mathcal{S}}
\newcommand{\prob}{\mathcal{P}}
\newcommand{\funl}[1]{\mathscr{#1}}      
\newcommand{\Fun}{\funl{F}}     
\DeclareMathAlphabet{\mathpzc}{OT1}{pzc}{m}{it}   
\newcommand{\setrn}[1]{\mathpzc{#1}}     
\newcommand{\ball}{\setrn{B}}
\newcommand{\IntS}{\setrn{S}}
\newcommand{\algo}[1]{\mbox{\texttt{#1}}}      
\newcommand{\dfo}{\algo{DFO}}
\newcommand{\newuoa}{\algo{NEWUOA}}
\newcommand{\uobyqa}{\algo{UOBYQA}}

\newcommand{\mnh}{\algo{MNH}}

\newcommand{\condor}{\algo{CONDOR}}
\newcommand{\cobyla}{\algo{COBYLA}}
\newcommand{\orbit}{\algo{ORBIT}}
\newcommand{\boosters}{\algo{BOOSTERS}}

\newcommand{\ESYMBS}{\algo{ESYMBS}}
\newcommand{\ESYMBP}{\algo{ESYMBP}}

\newcommand{\Real}{\mathbb{R}}
\newcommand{\Tran}[1]{#1^\mathrm{T}}

\newcommand{\pinv}[1]{#1^\dagger}
\newcommand{\Sob}[2]{H^#2}
\newcommand{\md}{\,\mathrm{d}}

\newcommand{\radius}{r}
\newcommand{\mean}{\mathrm{mean}}
\newcommand{\std}{\mathrm{std}}
\newcommand{\rstd}{\mathrm{rstd}}
\renewcommand{\theenumi}{\textnormal{\alph{enumi}.)}}   

\begin{document}

\title{Sobolev seminorm of quadratic functions with applications to 
derivative-free optimization
\thanks{Partially supported by Chinese NSF grants 10831006, 11021101, 
and CAS grant kjcx-yw-s7.}
}


\author{Zaikun Zhang}


\institute{Zaikun Zhang 
\at{Institute of Computational Mathematics and Scientific/Engineering 
Computing, Chinese Academy of Sciences, P.O. Box 2719, Beijing 100190, CHINA} \\
\email{zhangzk@lsec.cc.ac.cn}           
}

\date{Received: date / Accepted: date}

\maketitle

\begin{abstract}
  This paper studies the $\Sob{2}{1}$ Sobolev seminorm of quadratic functions. The research is motivated by the least-norm interpolation that is widely used in derivative-free optimization. We express the $\Sob{2}{1}$ seminorm of a quadratic function explicitly in terms of the Hessian and the gradient when the underlying domain is a ball. The seminorm gives new insights into least-norm interpolation. It clarifies the analytical and geometrical meaning of the objective function in least-norm interpolation. We employ the seminorm to study the extended symmetric Broyden update proposed by Powell. Numerical results show that the new thoery helps improve the performance of the update. Apart from the theoretical results, we propose a new method of comparing derivative-free solvers, which is more convincing than merely counting the numbers of function evaluations. 

\keywords{Sobolev seminorm \and Least-norm interpolation \and Derivative-free optimization \and Extended symmetric Broyden update}
 
 \subclass{90C56 \and 90C30 \and 65K05}
\end{abstract}

\section{Motivation and Introduction}
\label{motint}
Consider an unconstrained derivative-free optimization problem
\begin{equation}
\min_{x\in\Real^n}F(x),
  \label{dfoprob}
\end{equation}
where $F$ is a real-valued function whose derivatives are unavailable. Problems of this type have numerous applications. For instance, they have been employed to solve the helicopter rotor blade design problem \cite{booker1998managing,booker1998optimization,booker1999rigorous}, the ground water community problems \cite{hemker2006derivative,fowler2008comparison}, and the problems in biomedical imaging \cite{oeuvray2005trust,oeuvray2007new}. 

Many algorithms have been developed for problem (\ref{dfoprob}). For example, direct search methods (see Wright \cite{wright1996direct}, Powell \cite{powell1998direct}, Lewis, Torczon, and Trosset \cite{lewis2000direct}, and Kolda, Lewis, and Torczon \cite{kolda2003optimization} for reviews), line search methods without derivatives (see Stewart \cite{stewart1967modification} for a quasi-Newton method using finite difference; see Gilmore and Kelley \cite{gilmore1995implicit}, Choi and Kelley \cite{choi2000superlinear}, and Kelley \cite{kelley2002brief,kelley2011implicit} for Implicit Filtering method, a hybrid of quasi-Newton and grid-based search; see Diniz-Ehrhardt, Mart{\'\i}nez, and Rayd{\'a}n \cite{diniz2008derivative} for a derivative-free line search technique), and model-based methods (for instance, a method by Winfield \cite{winfield1973function}, \cobyla~by Powell \cite{COBYLA}, \dfo~by Conn, Scheinberg, and Toint \cite{CoScTo97a,Conn97recentprogress,Conn98DFOPractice}, \uobyqa~by Powell \cite{Powell2000uobyqa}, a wedge trust region method by Marazzi and Nocedal \cite{marazzi2002wedge}, \newuoa~and extensions by Powell \cite{Powell2004newuoa,Powell2008NEWUOAdev,powell2009bobyqa,powell2012beyond}, \condor~by Vanden Berghen and Bersini \cite{berghen2005condor}, \boosters~by Oeuvray and Bierlaire \cite{oeuvray2008derivative}, \orbit~by Wild, Regis, and Shoemaker \cite{wild2008orbit}, \mnh~by Wild \cite{SMW08MNH}, and a recent algorithm using sparse 
low degree model by Bandeira, Scheinberg, and Vicente \cite{bandeira2010computation}).
We refer the readers to the book by Brent \cite{brent1973algorithms}, the one by Kelley \cite{kelley1999iterative}, and the one by Conn, Scheinberg, and Vicente \cite{Conn:2009:IntroDFO} for extensive discussions and the references therein. 

Multivariate interpolation has been acting as a powerful tool in the design of derivative-free optimization methods. 
The following quadratic interpolation plays an important role in Conn and Toint \cite{CoTo96a}, Conn, Scheinberg, and Toint \cite{CoScTo97a,Conn97recentprogress,Conn98DFOPractice}, Powell \cite{PowellLFNU,Powell2004newuoa,Powell2008NEWUOAdev,powell2009bobyqa,powell2012beyond}, and Cust{\'o}dio, Rocha, and Vicente \cite{custodio2010incorporating}:
\begin{equation}
  \begin{split}
    \min_{Q\in\Qfs}&~\|\nabla^2Q\|_\mathrm{F}^2 + \sigma \|\nabla Q(x_0)\|_2^2\\
  \mathrm{s.t.}&~Q(x) = f(x),~~x\in \IntS,
  \label{lnint}
\end{split}
\end{equation}
where $\Qfs$ is the linear space of polynomials with degree at most two in $n$ variables, $x_0$ is a specific point in $\Real^n$, $\sigma$ is a nonnegative constant, $\IntS$ is a finite set of interpolation points in $\Real^n$, and $f$ is a function\footnote{Notice that $f$ is not always equal to the objective function $F$, which is the case in Powell \cite{PowellLFNU,Powell2004newuoa,Powell2008NEWUOAdev,powell2009bobyqa,powell2012beyond}. See Section \ref{lni} for details.} on $\IntS$. We henceforth refer to this type of interpolation as least-norm interpolation. 


The objective function of problem \textnormal{(\ref{lnint})} is interesting. It is easy to handle, and practice has proved it successful. The main purpose of this paper is to provide a new interpretation of it. Our tool is the $\Sob{2}{1}$ Sobolev seminorm, which is classical in PDE theory but may be rarely noticed in nonlinear optimization. We will give new insights into some basic questions about the objective function in (\ref{lnint}). For example, \emph{what is the exact analytical and geometrical meaning of this objective}? Why to combine the Frobenius norm of the Hessian and the $2$-norm of the gradient? What is the meaning of the parameters $x_0$ and $\sigma$? 

This paper is organized as follows. Section \ref{lni} introduces the applications of least-norm interpolation in derivative-free optimization. Section \ref{results} presents the main theoretical results. We first study the $\Sob{2}{1}$ seminorm of quadratic functions, and then employ the $\Sob{2}{1}$ seminorm to investigate least-norm interpolation. The questions listed above are answered in Section \ref{results}. Section \ref{ESB} applies the theory obtained in Section \ref{results} to study the extended symmetric Broyden update proposed by Powell \cite{powell2012beyond}. We show an easy and effective way to improve the performance of the update. Section \ref{secconclusion} concludes our discussion.

The main contributions of the paper lie in Section \ref{results} and Section \ref{ESB}. Besides the theoretical results, the numerical experience in Section \ref{ESB} is also a highlight. 
We propose a new method of testing derivative-free solvers by introducing statistics into the numerical experiments. See Subsection \ref{numerical} for details.

A remark on notation. We use the notation 
\begin{equation}
  \begin{split}
    \min_{x\in \setrn{X}}&~\psi(x)\\
    \textnormal{s.t.}&~\min_{x\in \setrn{X}} \phi(x)
  \end{split}
  \label{bilevel}
\end{equation}
for the bilevel programming problem 
\begin{equation}
  \begin{split}
    \min&~\psi(x)\\
    \textnormal{s.t.}&~x\in[\arg\min_{\xi\in \setrn{X}}\phi(\xi)].
  \end{split}
  \label{bilevelprog}
\end{equation}
\section{Least-Norm Interpolation in Derivative-Free Optimization}\label{lni}
Least-norm interpolation has successful applications in derivative-free optimization, especially in model-based algorithms. These algorithms typically follow trust-region methodology \cite{conn2000tr,powell2003trust}. On each iteration, a local model of the objective function is constructed, and then it is minimized within a trust-region to generate a trial step. The model is usually a quadratic polynomial constructed by solving an interpolation problem
\begin{equation}
  Q(x) = F(x),~~x\in\IntS,
  \label{int}
\end{equation}
where, as stated in problem (\ref{dfoprob}), $F$ is the objective function, and $\IntS$ is an interpolation set.  To determine a unique quadratic polynomial by problem (\ref{int}), the size of $\IntS$ needs to be at least $(n+1)(n+2)/2$, which is prohibitive when $n$ is big. Thus we need to consider underdetermined quadratic interpolation. In that case, a classical strategy to take up the remaining freedom is to minimize some functional subject to the  interpolation constraints, that is to solve 

\begin{equation}
  \begin{split}
    \min_{Q\in\Qfs}&~\Fun(Q)\\
  \mathrm{s.t.}&~Q(x) = F(x),~~x\in \IntS,
  \label{lfint}
  \end{split}
\end{equation}
$\Fun$ being some functional on $\Qfs$. Several existing choices of $\Fun$ lead to least-norm interpolation, directly or indirectly. Here we give some examples.

Conn and Toint \cite{CoTo96a} suggests the quadratic model that solves
\begin{equation}
  \begin{split}
    \min_{Q\in\Qfs}&~\|\nabla^2Q\|_\mathrm{F}^2 + \|\nabla Q(x_0)\|_2^2\\
  \mathrm{s.t.}&~Q(x) = F(x),~~x\in \IntS,
  \label{lnintCT}
\end{split}
\end{equation}
where $x_0$ is a specific point in $\Real^n$ (the center of the
trust-region they use). Problem (\ref{lnintCT}) is a least-norm interpolation problem with $\sigma=1$. 

Conn, Scheinberg, and Toint \cite{CoScTo97a,Conn97recentprogress,Conn98DFOPractice} builds a quadratic model by solving
\begin{equation}
  \begin{split}
    \min_{Q\in\Qfs}&~\|\nabla^2Q\|_\mathrm{F}^2\\
  \mathrm{s.t.}&~Q(x) = F(x),~~x\in \IntS.
  \label{lnintCST}
\end{split}
\end{equation}
It is considered as the best general strategy for the \dfo~algorithm to select a unique model from the pool of the possible interpolation models \cite{Conn98DFOPractice}. Wild \cite{SMW08MNH} also works with the model defined by (\ref{lnintCST}). Problem (\ref{lnintCST}) is a least-norm interpolation problem with $\sigma=0$. Conn, Scheinberg, and Vicente \cite{Conn:2009:IntroDFO} explains the motivation for (\ref{lnintCST}) by the following error bounds of quadratic interpolants.
\begin{theorem}\textnormal{\cite{Conn:2009:IntroDFO}}\label{thCSV}
  Suppose that the interpolation set $\IntS=\left\{y_0,~y_1,~\dots,~y_m\right\}$ ($m\ge n$) is contained in a ball $\ball(y_0,\radius)$ ($\radius>0$), and the matrix
  \begin{equation}
    L = \frac{1}{\radius}(y_1-y_0~ \cdots ~y_m-y_0)
  \end{equation}
  has rank $n$. If $F$ is continuously differentiable on $\ball(y_0,\radius)$, and $\nabla F$ is Lipschitz continuous on $\ball(y_0,\radius)$ with constant $\nu>0$, then for any quadratic function $Q$ satisfying the interpolation constraints (\ref{int}), it holds that
  \begin{equation}
  \|\nabla Q(x)-\nabla F(x)\|_2 \le \frac{5\sqrt{m}}{2}\|\pinv{L}\|_2(\nu + \|\nabla^2 Q\|_2)\radius,~~\forall x\in\ball(y_0,\radius),
    \label{ebg}
  \end{equation}
  and
  \begin{equation}
    |Q(x)- F(x)| \le \left(\frac{5\sqrt{m}}{2}\|\pinv{L}\|_2+\frac{1}{2}\right)(\nu + \|\nabla^2 Q\|_2)\radius^2, ~~ \forall x\in\ball(y_0,\radius).
    \label{ebf}
  \end{equation}
\end{theorem}
In light of Theorem \ref{thCSV}, minimizing some norm of $\nabla^2Q$ will help improve the approximation of gradient and function value. Notice that $\nabla^2Q$ appears in (\ref{ebg}--\ref{ebf}) with 2-norm rather than Frobenius norm.

Powell \cite{PowellLFNU,Powell2004newuoa,Powell2008NEWUOAdev} introduce the symmetric Broyden update to derivative-free optimization, and it is proposed to solve
\begin{equation}
  \begin{split}
    \min_{Q\in\Qfs}&~\|\nabla^2Q-\nabla^2Q_0\|_\mathrm{F}^2\\
  \mathrm{s.t.}&~Q(x) = F(x),~~x\in \IntS
  \label{lnintP1}
\end{split}
\end{equation}
to obtain a model for the current iteration, provided that $Q_0$ is the quadratic model used in the previous iteration. 
Problem (\ref{lnintP1}) is essentially a least-norm interpolation problem about \mbox{$Q-Q_0$}. 
The symmetric Broyden update is motivated by least change secant updates 
\cite{dennis1979least} in quasi-Newton methods. One particularly 
interesting advantage of the update is that, when $F$ is a quadratic function, 
$\nabla^2Q_+$ approximates $\nabla^2F$ better than $\nabla^2Q_0$ unless $\nabla^2Q_+=\nabla^2Q_0$, $Q_+$ being the model obtained by the update \cite{PowellLFNU,Powell2004newuoa,Powell2008NEWUOAdev}.

  
  Powell \cite{powell2012beyond} proposes the extended symmetric Broyden update by adding a first-order term to the objective function of problem (\ref{lnintP1}), resulting in 
\begin{equation}
  \begin{split}
    \min_{Q\in\Qfs}&~\|\nabla^2Q-\nabla^2Q_0\|_\mathrm{F}^2+\sigma\|\nabla Q(x_0)-\nabla Q_0(x_0)\|_2^2\\
  \mathrm{s.t.}&~Q(x) = F(x),~~x\in \IntS,
  \label{lnintP2}
\end{split}
\end{equation}
$x_0$ and $\sigma$ being specifically selected parameters, $\sigma$ nonnegative. Again, (\ref{lnintP2}) can be interpreted as a least-norm interpolation about \mbox{$Q-Q_0$}. Powell \cite{powell2012beyond} motivates (\ref{lnintP2}) by an algebraic example for which the symmetric Broyden update does not behave satisfactorily. We will study the extended symmetric Broyden update in Section \ref{ESB}.

  Apart from model-based methods, least-norm interpolation also has applications in direct search methods. Cust\'{o}dio, Rocha, and Vicente \cite{custodio2010incorporating} incorporates models defined by (\ref{lnintCST}) and (\ref{lnintP1}) into direct search. The authors attempt to enhance the performance of direct search methods by taking search steps based on these models. It is reported that (\ref{lnintCST}) works better for their method, and their procedure provides significant improvements to direct search methods of directional type.
  
  For more discussions about least-norm interpolation in derivative-free
  optimization, we refer the readers to Chapters 5 and 11 of Conn, Scheinberg, and Vicente \cite{Conn:2009:IntroDFO}.

\section{The $\Sob{2}{1}$ Sobolev Seminorm of Quadratic Functions}\label{results}
In Sobolev space theory \cite{adams1975sobolev,evans1998partial}, the $\Sob{2}{1}$ seminorm of a function $f$ over a domain $\Omega$ is defined as 
\begin{equation}
  |f|_{\Sob{2}{1}(\Omega)} = \left[\int_{\Omega}\|\nabla f(x)\|_2^2\md x\right]^{1/2}.
\end{equation}
In this section, we give an explicit formula for the $\Sob{2}{1}$ seminorm of quadratic functions when $\Omega$ is a ball, and accordingly present a new understanding of least-norm interpolation. We prove that least-norm interpolation essentially seeks the quadratic interpolant with minimal $\Sob{2}{1}$ seminorm over a ball. Theorem \ref{ThFormula}, Theorem \ref{insightnonzero}, and Theorem \ref{insightzero} are the main theoretical results of this paper.

\subsection{The $\Sob{2}{1}$ Seminorm of Quadratic Functions}\label{secformula}

The $\Sob{2}{1}$ seminorm of a quadratic function over a ball can be expressed explicitly in terms of its coefficients. We present the formula in the following theorem.

\begin{theorem}\label{ThFormula}
  Let $x_0$ be a point in $\Real^n$, $r$ be a positive number, and
  \begin{equation}
    \ball=\left\{x\in\Real^n:~\|x-x_0\|_2\leq r\right\}.
    \label{ball}
  \end{equation}
  Then for any $Q\in \Qfs$,
  \begin{equation}
    |Q|_{\Sob{2}{1}(\ball)}^2=V_nr^{n}\left[\frac{r^2}{n+2}\|\nabla^2 Q\|_\mathrm{F}^2+\|\nabla Q(x_0)\|_2^2\right],
    \label{formula}
  \end{equation}
  where $V_n$ is the volume of the unit ball in $\Real^n$.
\end{theorem}

\begin{proof}
 Without loss of generality, we assume that $x_0=0$. Let
 \begin{equation}
   g=\nabla Q(0),~~\textnormal{and}~~G=\nabla^2Q.
 \end{equation}
 Then
 \begin{equation}
   \label{snorm1}
   |Q|_{\Sob{2}{1}(\ball)}^2
   =\int_{\|x\|_2\leq r}\|Gx+g\|_2^2\md x
   =\int_{\|x\|_2\leq r}\left(\Tran{x}G^2x+2\Tran{x}Gg+\|g\|_2^2\right)\md x.
 \end{equation}
Because of symmetry, the integral of $\Tran{x}Gg$ is zero. Besides,
 \begin{equation}
   \int_{\|x\|_2\leq r}\|g\|_2^2\md x = V_nr^n\|g\|_2^2. 
   \label{intg}
 \end{equation}
 Thus we only need to find the integral of $\Tran{x}G^2x$. Now we assume that $G$ is diagonal (if not, apply a rotation). Denote the $i$-th diagonal entry of $G$ by $G_{(ii)}$, and the $i$-th coordinate of $x$ by $x_{(i)}$. Then
 \begin{equation}
   \int_{\|x\|_2\leq r}\Tran{x}G^2x\md x 
   =\int_{\|x\|_2\leq r}\left[\sum_{i=1}^nG_{(ii)}^2x_{(i)}^2\right]\md x
   =\sum_{i=1}^n\left[G_{(ii)}^2\int_{\|x\|_2\leq r}x_{(i)}^2\md x\right].
 \end{equation}
To finish the proof, we show that
\begin{equation}
  \label{intx}  \int_{\|x\|_2\leq r}x_{(i)}^2 \md x = \frac{V_nr^{n+2}}{n+2}.
\end{equation}
It suffices to justify (\ref{intx}) for the case $r=1$ and $n\ge 2$. First,
\begin{equation}
  \begin{split}
      &~\int_{\|x\|_2\le 1}x_{(i)}^2\md x\\
    =&~\int_{-1}^1u^2\md u\int_{\|v\|_2\le\sqrt{1-u^2}}\md v ~ ~ (v\in\Real^{n-1})\\
    =&~V_{n-1}\int_{-1}^1u^2(1-u^2)^{\frac{n-1}{2}}\md u\\
    =&~V_{n-1}\int_{-\frac{\pi}{2}}^{\frac{\pi}{2}}\sin^2\theta\cos^n\theta\md \theta,
\end{split}
  \label{intxsincos}
\end{equation}
and, similarly,
\begin{equation}
  \label{vsincos}V_n = V_{n-1}\int_{-\frac{\pi}{2}}^{\frac{\pi}{2}} \cos^n\theta\md \theta.  
\end{equation}
Second, integration by parts shows that
\begin{equation}
  \label{intbyp}\int_{-\frac{\pi}{2}}^{\frac{\pi}{2}} \cos^{n+2}\theta\md \theta = (n+1)\int_{-\frac{\pi}{2}}^{\frac{\pi}{2}} \sin^2\theta\cos^n\theta\md \theta.  
\end{equation}
Now it is easy to obtain (\ref{intx}) from (\ref{intxsincos}-\ref{intbyp}).
\qed
\end{proof}

Theorem \ref{ThFormula} tells us that the $\Sob{2}{1}$ seminorm of a quadratic function $Q$ over a ball $\ball$ is closely related to a combination of $\|\nabla^2Q\|_{\mathrm{F}}^2$ and $\|\nabla Q(x_0)\|_2^2$, $x_0$ being the center of $\ball$, and the combination coefficients being determined by the radius of $\ball$. This result is interesting for two reasons. First, it enables us to measure the overall magnitude of the gradient over a ball. This is nontrivial for a quadratic function, since its gradient is not constant. Second, it clarifies the analytical and geometrical meaning of combining the Frobenius norm of the Hessian and the 2-norm of the gradient, and enables us to select the combination coefficients according to the geometrical meaning.

\subsection{New Insights into Least-Norm Interpolation}\label{relation}
The $\Sob{2}{1}$ seminorm provides a new angle of view to understand least-norm interpolation. For convenience, we rewrite the least-norm interpolation problem here and call it the problem \ref{lnintr}:
\begin{equation}
  \label{lnintr}\tag{$\textnormal{P}_1(\sigma)$}
  \begin{split}
    \min_{Q\in\Qfs}&~\|\nabla^2Q\|_\mathrm{F}^2 + \sigma\|\nabla Q(x_0)\|_2^2\\
  \mathrm{s.t.}&~Q(x) = f(x),~~x\in \IntS.
\end{split}
\end{equation}
We assume that the interpolation system
\begin{equation}
  Q(x) = f(x),~~x\in\IntS
  \label{intsys}
\end{equation}
is consistent on $\Qfs$, namely
\begin{equation}
  \left\{Q\in\Qfs:~Q(x)=f(x)\textnormal{~for~every~}x\in\IntS\right\}\neq \emptyset.
  \label{consist}
\end{equation}
With the help of Theorem \ref{ThFormula}, we see that the problem \ref{lnintr} is equivalent to the problem
\begin{equation}
  \label{lsnint}
  \tag{$\textnormal{P}_2(r)$}
  \begin{split}
    \min_{Q\in\Qfs}&~|Q|_{\Sob{2}{1}(\ball_r)}\\
  \mathrm{s.t.}&~Q(x) = f(x),~~x\in \IntS
\end{split}
\end{equation}
in some sense. The purpose of this subsection is to clarify the equivalence. 

When $\sigma>0$, the equivalence is easy and we state it as follows.
\begin{theorem}\label{insightnonzero}
  If $\sigma>0$, then the least-norm interpolation problem \ref{lnintr} is equivalent to the problem \ref{lsnint} with $r = \sqrt{(n+2)/\sigma}$.
\end{theorem}

It turns out that the least-norm interpolation problem with positive $\sigma$ essentially seeks the interpolant with minimal $\Sob{2}{1}$ seminorm over a ball. The geometrical meaning of $x_0$ and $\sigma$ is clear now: $x_0$ is the center of the ball, and $\sqrt{(n+2)/\sigma}$ is the radius.

Now we consider the least-norm interpolation problem with $\sigma=0$. This case is particularly interesting, because it appears in several practical algorithms \cite{Conn98DFOPractice,Powell2004newuoa,SMW08MNH}. Since \ref{lnintr} may have multiple solutions when $\sigma=0$, we modify the definition of \ref{lfnint}: from now on, \ref{lfnint} is defined to be the bilevel least-norm interpolation problem   
 \begin{equation}
\label{lfnint}
\tag{$\textnormal{P}_1(0)$}
\begin{split}
  \min_{Q\in\Qfs}&~\|\nabla Q(x_0)\|_2\\
  \textnormal{s.t.}&~\min_{Q\in\Qfs}~\|\nabla^2 Q\|_\mathrm{F}\\
&~~~\textnormal{s.t.}~~Q(x)=f(x),~~x\in\IntS.
\end{split}
\end{equation}
The new definition is reasonable, because the following proposition holds after the modification.
\begin{proposition}
  For each $\sigma\ge 0$, the problem \ref{lnintr} has a unique solution $Q_\sigma$. Moreover, $Q_\sigma$ converges\footnote{We define the convergence on $\Qfs$ to be the convergence of coefficients.} to $Q_0$ when $\sigma$ tends to $0^+$.
  \label{converge}
\end{proposition}
\begin{proof}
  The uniqueness of the solution is simple.
  First, the uniqueness of the Hessian and gradient is guaranteed by the convexity
  of the objective functional(s); then with the help of any one of the interpolation 
  constraints, we find that the constant term is also unique. 

  In light of the uniqueness stated above, the convergence of $\left\{ Q_\sigma \right\}$ 
  is a corollary of the classical
  penalty function theory in nonlinear optimization (see, for instance, Theorem 12.1.1 
  of Fletcher \cite{Fletcher:1987:PMO}). \qed
\end{proof}

Note that Proposition \ref{converge} implies the problem \ref{lsnint} has a unique solution for each positive $r$. Now we can state the relation between the problems \ref{lfnint} and \ref{lsnint}.
\begin{theorem}\label{insightzero}
  When $r$ tends to infinity, the solution of \ref{lsnint} converges to the solution of \ref{lfnint}.
\end{theorem}

Theorem \ref{insightzero} indicates that, when $\sigma=0$, the least-norm interpolation problem seeks the interpolant with minimal $\Sob{2}{1}$ seminorm over $\Real^n$ in the sense of limit. Thus Theorem \ref{insightzero} is the extension of Theorem \ref{insightnonzero} to the case $\sigma=0$.

The questions listed in Section \ref{motint} can be answered now. The meaning of the objective function in least-norm interpolation is to minimize the $\Sob{2}{1}$ seminorm of the interpolant over a ball. The reason to combine the Frobenius norm of the Hessian and the 2-norm of the gradient is to measure the $\Sob{2}{1}$ seminorm of the interpolant. The parameters $x_0$ and $\sigma$ determines the ball where the $\Sob{2}{1}$ seminorm is calculated, $x_0$ being the center and $\sqrt{(n+2)/\sigma}$ being the radius. When $\sigma=0$, we can interprete least-norm interpolation in the sence of limit. 

\section{On the Extended Symmetric Broyden Update} \label{ESB}
In this section, we employ the $\Sob{2}{1}$ seminorm to study the extended symmetric Broyden update proposed by Powell \cite{powell2012beyond}. As introduced in Section \ref{lni}, the update defines $Q_+$ to be the solution of 
\begin{equation}
  \begin{split}
    \min_{Q\in\Qfs}&~\|\nabla^2Q-\nabla^2Q_0\|_\mathrm{F}^2+\sigma\|\nabla Q(x_0)-\nabla Q_0(x_0)\|_2^2\\
  \mathrm{s.t.}&~Q(x) = F(x),~~x\in \IntS,
  \label{lnintBB}
\end{split}
\end{equation}
$F$ being the objective function, and $Q_0$ being the model used in the previous trust-region iteration. When $\sigma=0$, it is the symmetric Broyden update studied by Powell \cite{PowellLFNU,Powell2004newuoa,Powell2008NEWUOAdev}. We focus on the case $\sigma>0$. In Subsection \ref{intupd}, we interpret the update with the $\Sob{2}{1}$ seminorm. In Subsection \ref{choice}, we discuss the choices of $x_0$ and $\sigma$ in the update, and show how to choose $x_0$ and $\sigma$ according to their geometrical meaning. In Subsection \ref{numerical}, we test our choices of $x_0$ and $\sigma$ through numerical experiments. It is worth mentioning that, the numerical experiments in Subsection \ref{numerical} are designed in a special way in order to reduce the influence of computer rounding errors and obtain more convincing results. 

\subsection{Interpret the Update with the $\Sob{2}{1}$ Seminorm} \label{intupd}
According to Theorem \ref{insightnonzero}, problem (\ref{lnintBB}) is equivalent to
\begin{equation}
  \begin{split}
    \min_{Q\in\Qfs}&~|Q-Q_0|_{\Sob{2}{1}(\ball)}\\
  \mathrm{s.t.}&~Q(x) = F(x),~~x\in \IntS,
  \end{split}
  \label{lnintBBh}
\end{equation}
where
\begin{equation}\label{ballsig}
  \ball=\left\{x\in\Real^n:~\|x-x_0\|_2\leq \sqrt{(n+2)/\sigma}\right\}.
\end{equation}
Thus the extended symmetric Broyden update seeks the closest quadratic model to $Q_0$ subject to the interpolation constraints, distance being measured by the seminorm $|\cdot|_{\Sob{2}{1}(\ball)}$. 

Similar to the symmetric Broyden update, the extended symmetric Broyden update enjoys a very good approximation property when $F$ is a quadratic function, as stated in Proposition \ref{decreaseh}.  
\begin{proposition}\label{decreaseh}
  If $F$ is a quadratic function, then the solution $Q_+$ of problem (\ref{lnintBB}) satisfies
\begin{equation}\label{decreaseh1}
  |Q_+-F|_{\Sob{2}{1}(\ball)}^2 = |Q_0-F|_{\Sob{2}{1}(\ball)}^2 - |Q_+-Q_0|_{\Sob{2}{1}(\ball)}^2,
\end{equation}
where $\ball$ is defined as (\ref{ballsig}).
\end{proposition}
\begin{proof}
Let $Q_t = Q_+ + t(Q_+-F)$, $t$ being any real number. Then $Q_t$ is a quadratic function interpolating $F$ on $\IntS$, as $F$ is quadratic and $Q_+$ interpolates it. The optimality of $Q_+$ implies that the function
  \begin{equation}
    \varphi(t) = |Q_t- Q_0|_{\Sob{2}{1}(\ball)}^2
  \end{equation}
attains its minimum when $t$ is zero. Expanding $\varphi(t)$, we obtain
\begin{equation}
  \varphi(t) = t^2|Q_+- F|_{\Sob{2}{1}(\ball)}^2+2t\int_{\ball}\Tran{[\nabla (Q_+-Q_0)(x)]}[\nabla (F-Q_0)(x)]\md x+|Q_+- Q_0|_{\Sob{2}{1}(\ball)}^2.
\end{equation}
Hence
\begin{equation}
  \int_{\ball}\Tran{[\nabla (Q_+-Q_0)(x)]}[\nabla (F-Q_0)(x)]\md x = 0.
\end{equation}
Then we obtain (\ref{decreaseh1}) by considering $\varphi(-1)$.\qed 
\end{proof}
In light of Theorem \ref{ThFormula}, equation (\ref{decreaseh1}) is equivalent to equation (1.9) of Powell \cite{powell2012beyond}. Notice that equation (\ref{decreaseh1}) implies
\begin{equation}\label{decreaseg}
  \int_{\ball}\|\nabla Q_+(x)-\nabla F(x)\|_2^2\md x \le \int_{\ball}\|\nabla Q_0(x)-\nabla F(x)\|_2^2\md x.
\end{equation}
In other words, $\nabla Q_+$ approximates $\nabla F$ on $\ball$ better than $\nabla Q_0$ unless $\nabla Q_+ = \nabla Q_0$. 
\subsection{Choices of $x_0$ and $\sigma$} \label{choice}
Now we turn our attention to the choices of $x_0$ and $\sigma$ in the update. Recall that the purpose of the update is to construct a model for a trust-region subproblem. As in classical trust-region methods, the trust region is available before the update is applied, and we suppose it is 
\begin{equation}
  \left\{x: \|x-\bar{x}\|_2 \le \Delta\right\},
  \label{trustregion}
\end{equation}
the point $\bar{x}$ being the trust-region center, and the positive number $\Delta$ being the trust-region radius.

Powell \cite{powell2012beyond} chooses $x_0$ and $\sigma$ by exploiting the Lagrange functions of the interpolation problem (\ref{lnintBB}). Suppose that $\IntS = \left\{y_0,~y_1,~\dots,~y_m \right\}$, then the $i$-th ($i = 0,~1,~\dots,~m$) Lagrange function of problem (\ref{lnintBB}) is defined to be the solution of 
\begin{equation}
  \begin{split}
    \min_{l_i\in\Qfs}&~\|\nabla^2l_i\|_{\textrm{F}}^2+\sigma\|\nabla l_i\|_2^2\\
    \mathrm{s.t.}&~l_i(y_j) = \delta_{ij},~~j = 0,~1,~\dots,~m,
  \end{split}
  \label{laglnintBB}
\end{equation}
$\delta_{ij}$ being the Kronecker delta. In the algorithm of Powell \cite{powell2012beyond}, $\IntS$ is maintained in a way so that $Q_0$ interpolates $F$ on $\IntS$ except one point, say $y_0$ without loss of generality. Then the interpolation constraints can be rewritten as
\begin{equation}
  Q(y_j)-Q_0(y_j) = \left[F(y_0)-Q_0(y_0)\right]\delta_{0j},~~j=0,~1,~\dots,~m.
\end{equation}
Therefore the solution of problem (\ref{lnintBB}) is
\begin{equation}
  Q_+ = Q_0+\left[F(y_0)-Q_0(y_0)\right]l_0.
\end{equation}
Hence $l_0$ plays a central part in the update. Thus Powell \cite{powell2012beyond} chooses $x_0$ and $\sigma$ by examining $l_0$. Powell shows that, in order to make sure $l_0$ behaves well, the ratio $\|x_0-\bar{x}\|_2/\Delta$ should not be much larger than one, therefore $x_0$ is set to be $\bar{x}$ throughout the calculation. The choice of $\sigma$ is a bit complicated. The basic idea is to balance $\|\nabla^2 l_0\|_{\textrm{F}}^2$ and $\sigma\|\nabla l_0(x_0)\|_2^2$. Thus $\sigma$ is set to $\eta/\xi$, where $\eta$ and $\xi$ are estimates of the magnitudes of $\|\nabla^2 l_0\|_{\textrm{F}}^2$ and  $\|\nabla l_0(x_0)\|_2^2$, respectively. See Powell \cite{powell2012beyond} for details.

In Subsection \ref{intupd}, we have interpreted the extended symmetric
Broyden update with the ${\Sob{2}{1}}$ seminorm. The new interpretation
enables us to choose $x_0$ and $\sigma$ in a geometrical way. According
to problem (\ref{lnintBBh}), choosing $x_0$ and $\sigma$ is equivalent
to choosing the ball $\ball$. Let us think about the role that $\ball$
plays in the update. First, $\ball$ is the region where the update tries
to preserve information from $Q_0$, by minimizing the change with
respect to the seminorm $|\cdot|_{\Sob{2}{1}(\ball)}$. Second, $\ball$
is the region where the update tends to improve the model, as suggested
by the facts (\ref{decreaseh1}) and (\ref{decreaseg}) for quadratic
objective function. Thus we should choose $\ball$ to be a region where
the behavior of the new model is important to us. It may seem adequate
to pick $\ball$ to be the trust region (\ref{trustregion}). But we
prefer to set it bigger, because the new model $Q_+$ will influence its
successors via the least change update, and thereby influence subsequent
trust-region iterations. Thus it is myopic to consider only the current
trust region, and a more sensible choice is to let $\ball$ be the ball
\mbox{$\left\{x: \|x-\bar{x}\|_2\le M\Delta\right\}$}, $M$ being
a positive number bigger than one\footnote{Since 
the trust-region radii of the subsequent iterations 
vary dynamically according to the degree of success, there should
be more elaborate ways to define the domain
$\ball$ adaptively.
The definition given here might be the simplest one, and is enough for
our experiment.}. 
Moreover, we find in practice that it is helpful to require \mbox{$\ball\supset\IntS$}. Thus our choice of $\ball$ is
\begin{equation}
\left\{x: \|x-\bar{x}\|_2\le r\right\},
\end{equation}
where 
\begin{equation}
  r = \max\left\{M\Delta, \max_{x\in\IntS}\|x-\bar{x}\|_2\right\}.
  \label{chooser}
\end{equation}
Consequently, we set 
\begin{equation}
  x_0=\bar{x},~~\textnormal{and}~~\sigma=\frac{n+2}{r^2}.
  \label{choosexs}
\end{equation}
Therefore our choice of $x_0$ coincides with Powell's, but the choice of $\sigma$ is different.
\subsection{Numerical Results} \label{numerical}

With the help of the $\Sob{2}{1}$ seminorm, we have proposed a new method of choosing $\sigma$ for the extended symmetric Broyden update. In this subsection we test the new $\sigma$ through numerical experiments, and make comparison with the one by Powell \cite{powell2012beyond}. 

Powell \cite{powell2012beyond} implements a trust-region algorithm based on the extended symmetric Broyden update in Fortran 77.
We modify this code to get two solvers as follows for comparison.
\begin{enumerate}
  \item \ESYMBP: $\sigma$ is chosen according to Powell \cite{powell2012beyond}, as described in the second paragraph of Subsection \ref{choice}.
  \item \ESYMBS: $\sigma$ is chosen according to (\ref{chooser}--\ref{choosexs}). We set $M = 10$ in (\ref{chooser}).
  \end{enumerate}
  In the code of Powell \cite{powell2012beyond}, the size of the interpolation set $\IntS$ is kept unchanged throughout the computation, and the user can set it to be any integer between $n+2$ and $(n+1)(n+2)/2$. We chose $2n+1$ as recommended. 
The terminate criterion of this code is determined by a parameter
named \texttt{RHOEND}. \texttt{RHOEND} acts as the final 
trust-region radius in the code, and its magnitude usually agrees with
the precision of the computational solution. 
We tested \ESYMBP and \ESYMBS with \texttt{RHOEND} $ = 10^{-2}, 10^{-4}$,
and $10^{-6}$, to observe the performance of the solvers under different
requirements of solution precision. 
  The code of Powell \cite{powell2012beyond} is capable of solving problems with box constraints. We set the bounds to infinity since we are focusing on unconstrained problems. 

  We assume that evaluating the objective function is the most expensive part for optimization without derivatives. Thus we might compare the performance of different derivative-free solvers by simply counting the numbers of function evaluations they use until termination, provided that they have found the same minimizer to high accuracy. However, this is inadequate in practice, because, according to our numerical experiences, computer rounding errors could substantially influence the number of function evaluations. 
  Powell \cite{Powell2004newuoa} presents a very interesting example for this 
  phenomenon. In the example, the computational result of a test problem
  was influenced dramatically by the order of variables. But the code under
  test (\newuoa) was mathematically independent of variables' order. Thus 
  the differences in the results are due entirely to computer rounding errors.
  Please see Section 8 of Powell  \cite{Powell2004newuoa} for details of the example.

  We test the solvers \ESYMBS~and \ESYMBP~with the following method, in order to observe their performance in a relatively reliable way, and to inspect their stability with respect to 
  computer rounding errors. The basic idea is from the numerical experiments of Powell \cite{powell2012beyond}. Given a test problem $\prob$ with objective function $F$ and starting point $\hat{x}$, we randomly generate $N$ permutation matrices $P_i$ ($i=1,~2,~\dots,~N$), and let 
  \begin{equation}
    F_i(x) = F(P_ix),~~\hat{x}_i = P_i^{-1}\hat{x},~~i = 1,~2,~\dots,~N.
  \end{equation} 
  Then we employ the solvers to minimize $F_i$ starting form $\hat{x}_i$. For solver $\solver$, we obtain a vector $\#F=(\#F_1,~\#F_2,~\dots,~\#F_N)$, $\#F_i$ being the number of function evaluations required by $\solver$ for $F_i$ and $\hat{x}_i$. If the code of $\solver$
  is mathematically independent of variables' order (it is the case for \ESYMBS~and \ESYMBP),
  then all the entries of $\#F$ are identical in theory. However, it
  is not the case in practice, and the differences are completely made by
  computer rounding errors. 
  We compute the mean value and the standard deviation of vector $\#F$:
  \begin{equation}
    \mean(\#F) = \frac{1}{N}\sum_{i=1}^N\#F_i,
  \end{equation}
  and 
  \begin{equation}
    \std(\#F) = \sqrt{\frac{1}{N}\sum_{i=1}^N\left[\#F_i-\mean(\#F)\right]^2}.
  \end{equation}
When $N$ is reasonably large, $\mean(\#F)$ estimates the average performance of solver $\solver$ on problem $\prob$, and $\std(\#F)$ reflects the stability of solver $\solver$ with respect to computer rounding errors when solving problem $\prob$. We may also compute the relative standard deviation of $\#F$, namely, 
  \begin{equation}
    \rstd(\#F) = \frac{\std(\#F)}{\mean(\#F)},
    \label{rstd}
  \end{equation}
to obtain a normalized measure of stability. 
By comparing these statistics, we can reasonably assess the solvers under consideration. We use $N=10$ in practice.

With the method addressed above, we tested \ESYMBP~and \ESYMBS~on 30 unconstrained 
test problems with alterable dimension. The problems are from CUTEr \cite{cuter}
and Powell \cite{Powell2004newuoa}, with names listed in Table \ref{tab:probs}.

\begin{table}
    \centering
    \caption{Name of Test Problems}
    \label{tab:probs}
    \begin{tabular}{lllll}
\hline\noalign{\smallskip}
ARGLINA& ARGLINB& ARGLINC& ARWHEAD& BDQRTIC\\
BROYDN3D& BRYBND& CHROSEN& COSINE& CURLY10\\ 
CURLY20& CURLY30& DQRTIC& EDENSCH& EG2\\ 
ENGVAL1& FLETCBV2&  FREUROTH& GENBROWN& INTEGREQ\\
NONCVXUN& PENALTY1& PENALTY2& POWER& SCHMVETT\\ 
SPARSINE& SPARSQUR& SPHRPTS& TOINTGSS& VARDIM\\
\hline\noalign{\smallskip}
    \end{tabular}
\end{table}


For each problem, we solved it for dimensions 6, 8, 10, 12, 14, and 16. 
We set the maximal number of
function evaluations to be 1000. 
The experiments were carried out on a Thinkpad T400 laptop with Linux 2.6.33, and the compiler was gfortran in GCC 4.3.3.

We visualize the results of our experiments by presenting the performance profiles \cite{dolan2002benchmarking} of $\mean(\#F)$ in Fig. \ref{2pfmean}, \ref{4pfmean}, and \ref{6pfmean}. The same as traditional performance profiles, the higher means the better. 
The profiles suggest that, in our experiments, \ESYMBS~performed better 
than \ESYMBP~in the sense of number of function evaluations. Moreover, 
\ESYMBS~also showed better stability than \ESYMBP, with respect to computer rounding 
errors. To make the comparison, we present the performance profiles
of $\rstd$ in Fig. \ref{2pfrstd}, \ref{4pfrstd}, and \ref{6pfrstd}. 
As before, the higher means the better.
Thus the profiles support the conclusion that \ESYMBS~performed stabler than  \ESYMBP~with respect
to computer rounding errors. 

From the results of our experiments, we conclude that the $\Sob{2}{1}$ seminorm
has suggested a better way of parameter selection for extended symmetric Broyden 
update, compared with the method given in Powell \cite{powell2012beyond}. The new 
selection reduces the number of function evaluations in the resultant 
algorithm, and improves the solver's stability with respect to computer rounding errors.
Moreover, the new selection is much simpler and easier to understand than
the one in Powell \cite{powell2012beyond}. 


\begin{figure}[htbp] 
\centering 
\includegraphics[width=0.7\textwidth]{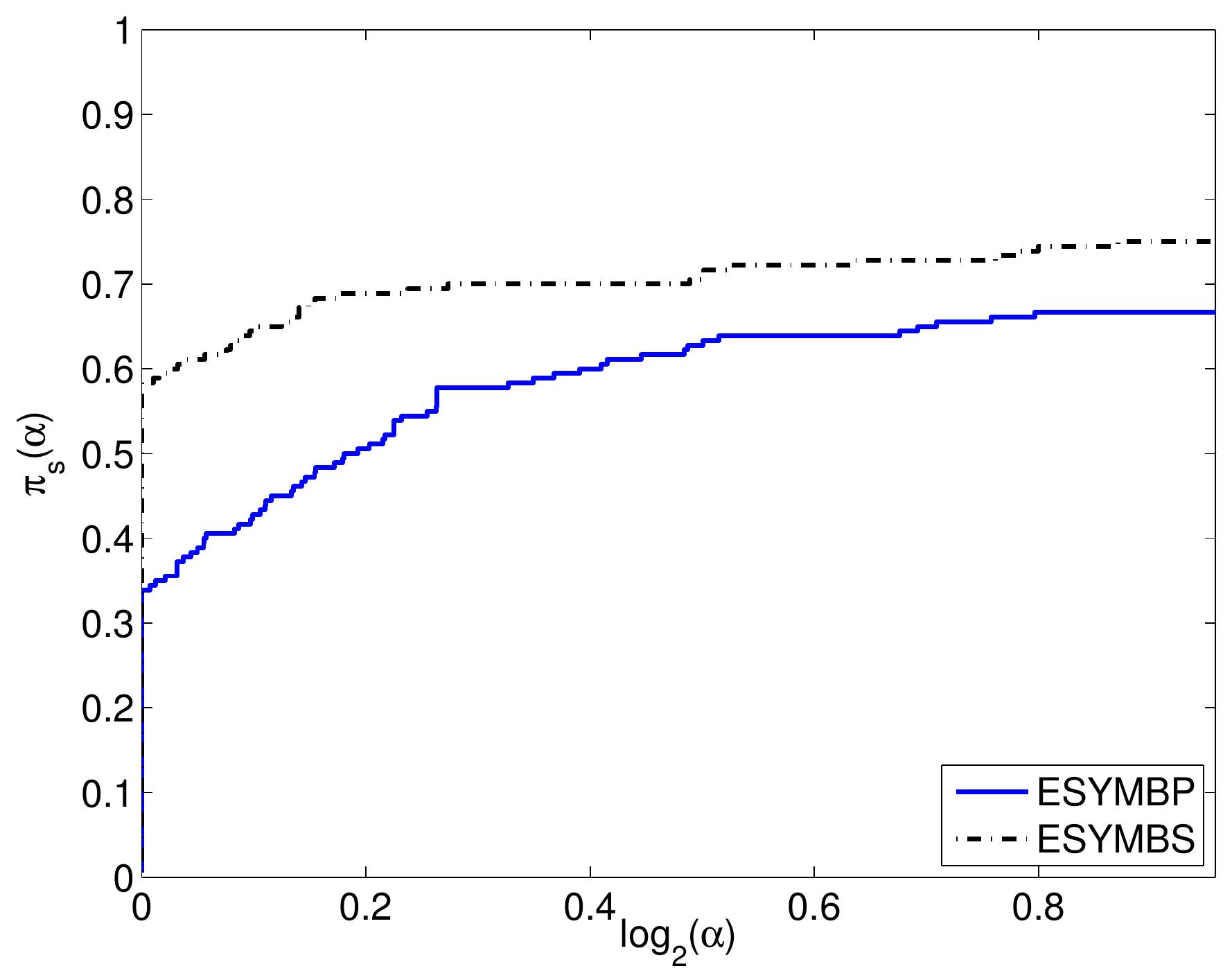} 
\caption{Performance Profile of $\mean(\#F)$ (\texttt{RHOEND}$=10^{-2}$)} \label{2pfmean} 
\end{figure}
\begin{figure}[htbp] 
\centering 
\includegraphics[width=0.7\textwidth]{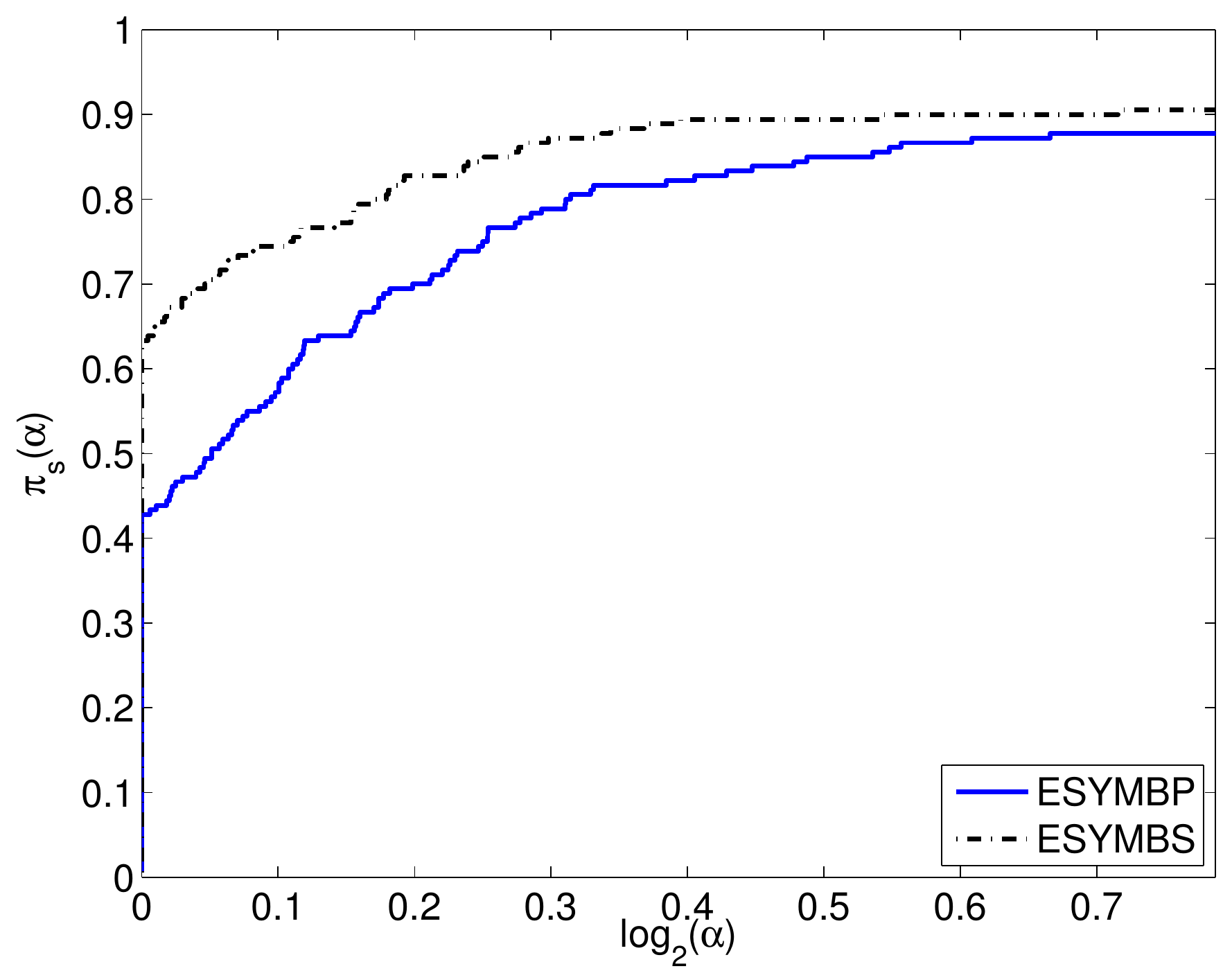} 
\caption{Performance Profile of $\mean{(\#F)}$ (\texttt{RHOEND}$=10^{-4}$)} \label{4pfmean} 
\end{figure}
\begin{figure}[htbp] 
\centering 
\includegraphics[width=0.7\textwidth]{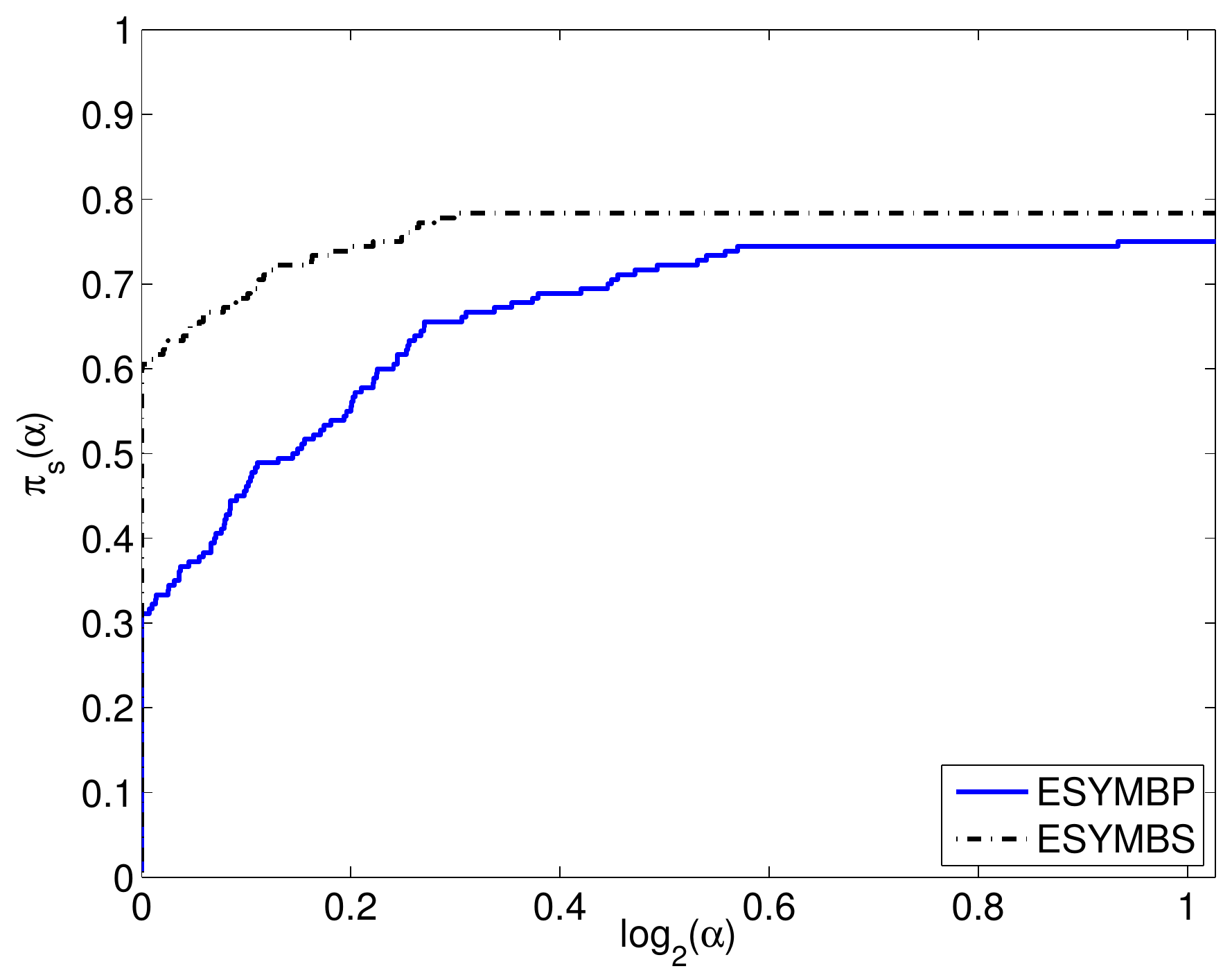} 
\caption{Performance Profile of $\mean{(\#F)}$ (\texttt{RHOEND}$=10^{-6}$)} \label{6pfmean} 
\end{figure}

\begin{figure}[htbp] 
\centering 
\includegraphics[width=0.7\textwidth]{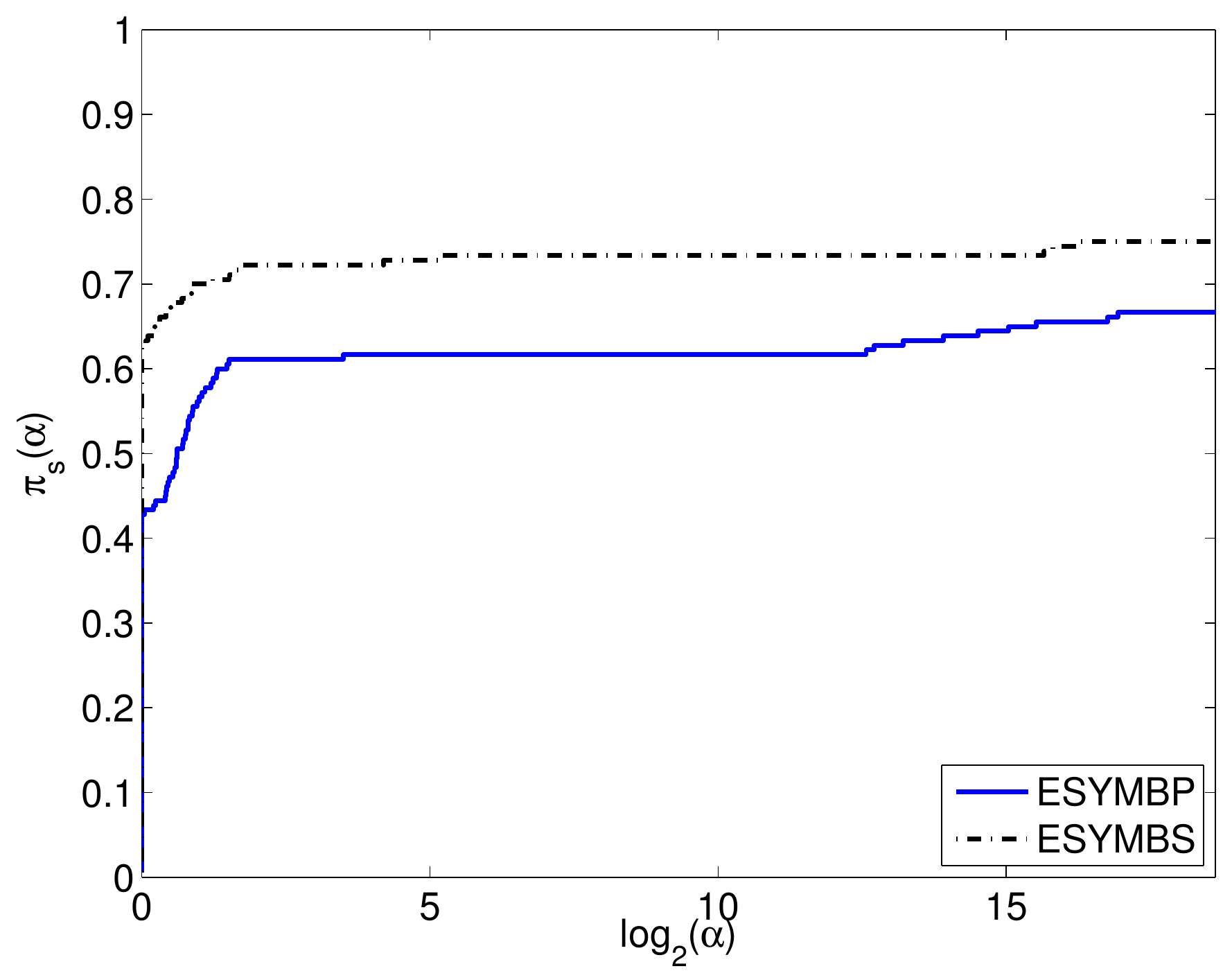} 
\caption{Performance Profile of $\rstd{(\#F)}$ (\texttt{RHOEND}$=10^{-2}$)} \label{2pfrstd} 
\end{figure}
\begin{figure}[htbp] 
\centering 
\includegraphics[width=0.7\textwidth]{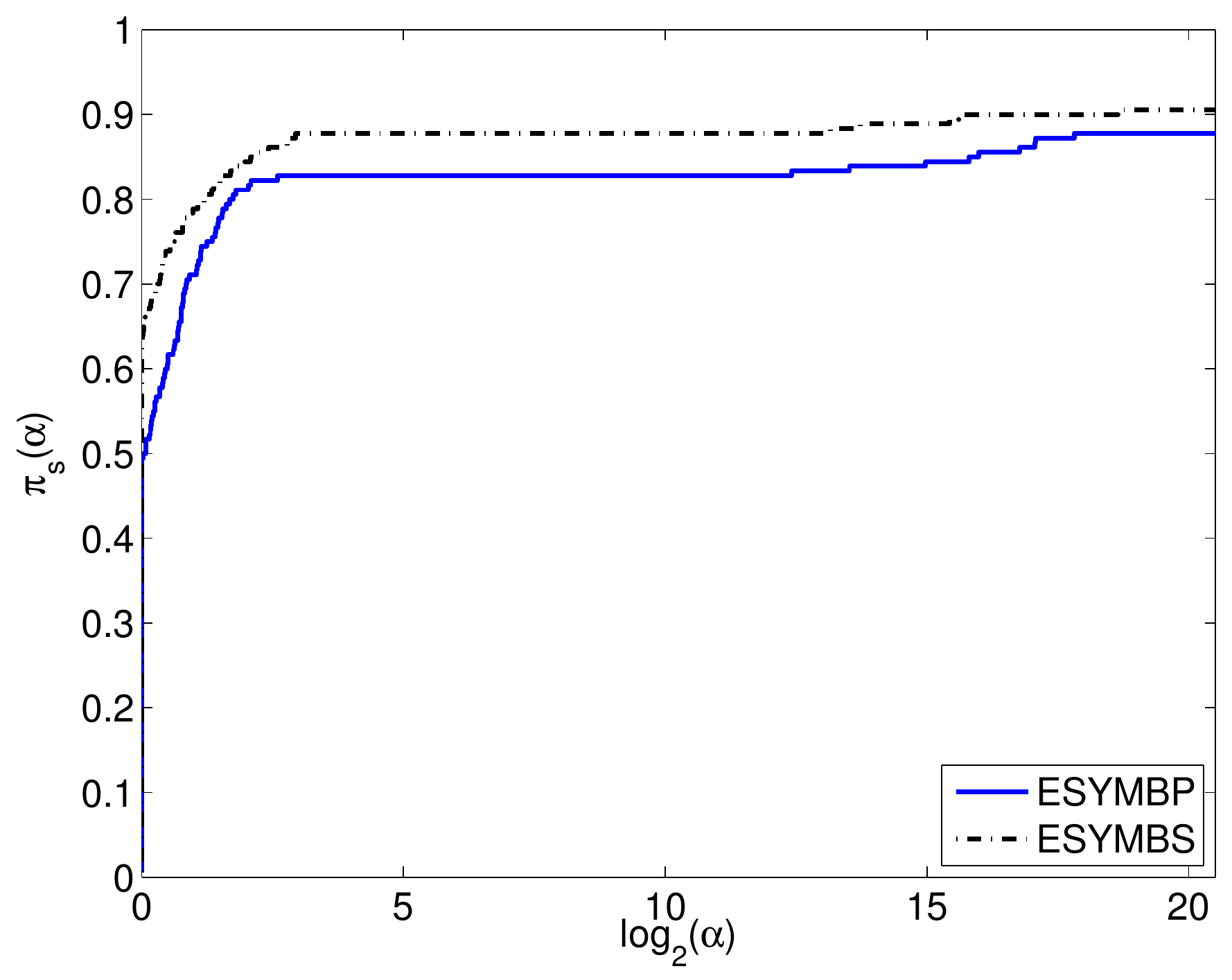} 
\caption{Performance Profile of $\rstd{(\#F)}$ (\texttt{RHOEND}$=10^{-4}$)} \label{4pfrstd} 
\end{figure}
\begin{figure}[htbp] 
\centering 
\includegraphics[width=0.7\textwidth]{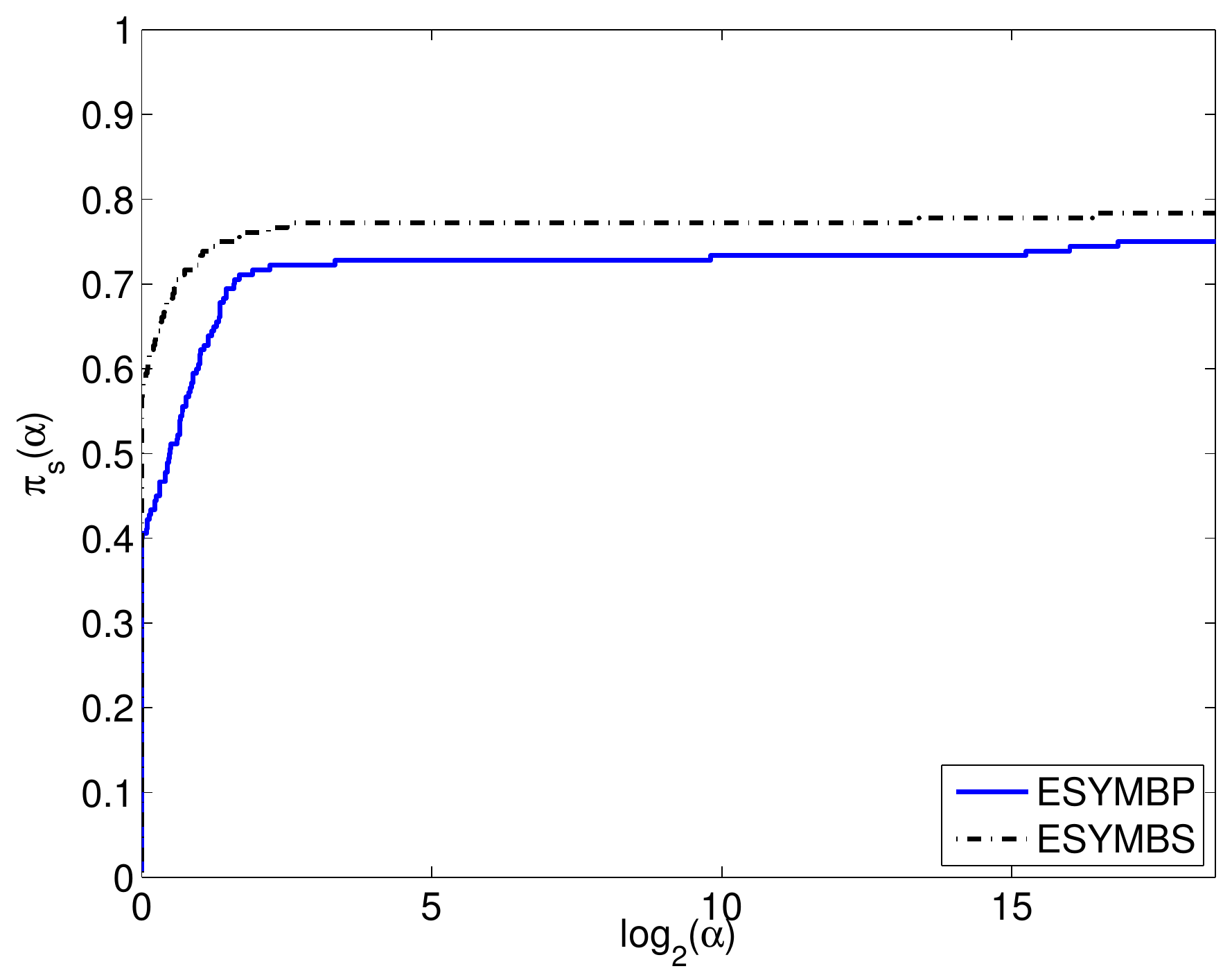} 
\caption{Performance Profile of $\rstd{(\#F)}$ (\texttt{RHOEND}$=10^{-6}$)} \label{6pfrstd} 
\end{figure}

\subsection{A Remark on Extended Symmetric Broyden Update}
It is interesting to ask whether the extended symmetric Broyden update can bring improvements to the original version in derivative-free optimization. Powell \cite{powell2012beyond} studies this problem, and it turns out that the extended version rarely reduced the number of function evaluations. We have to point out that it is also the case when $\sigma$ is selected with our method.  However, some numerical experiments of  Powell \cite{powell2012beyond} suggested that the extended version might help to improve the convergence rate. It would be desirable to theoretically investigate the local convergence properties of the updates. We think it is still too early to draw any conclusion on the extended symmetric Broyden update, unless we could find strong theoretical evidence. 

\section{Conclusions}\label{secconclusion}

Least-norm interpolation has wide applications in derivative-free optimization. Motivated by  its objective function, we have studied the $\Sob{2}{1}$ Sobolev seminorm of quadratic functions. We gives an explicit formula for the $\Sob{2}{1}$ seminorm of quadratic functions over balls. It turns out that least-norm interpolation essentially seeks a quadratic interpolant with minimal $\Sob{2}{1}$ seminorm over a ball. Moreover, we find that the parameters $x_0$ and $\sigma$ in the objective function determines the center and radius of the ball, respectively. These insights further our understanding of least-norm interpolation. 

With the help of the new observation, we have studied the extended symmetric Broyden update (\ref{lnintBB}) proposed by Powell \cite{powell2012beyond}. Since the update calculates the change to the old model by a least-norm interpolation, we can interpret it with our new theory. We have discussed how to choose $x_0$ and $\sigma$ for the update according to their geometrical meaning. This discussion leads to the same $x_0$ used by Powell \cite{powell2012beyond}, and a very easy way to choose $\sigma$. According to our numerical results, the new $\sigma$ works better than the one proposed by Powell \cite{powell2012beyond}.

Our method of designing numerical tests seems reasonable. It provides a 
more reliable measurement of computation cost than simply counting the 
numbers of function evaluations, and meanwhile reflects the stability of
solvers with respect to computer rounding errors. We hope this method can 
lead to better benchmarking of derivative-free solvers.  

\begin{acknowledgements}
  The author thanks Professor \mbox{M. J. D. Powell} for providing references and codes, for kindly answering the author's questions about the excellent software \newuoa, and for the helpful discussions. Professor Powell helped improve the proofs of Theorem \ref{ThFormula} and Proposition \ref{converge}. The author owes much to his supervisor Professor Ya-Xiang Yuan for the valuable directions and suggestions. This work was partly finished during a visit to Professor Klaus Schittkowski at Universit\"{a}t Bayreuth in 2010, where the author has received pleasant support for research. The author thanks Humboldt Fundation for the financial assistance during this visit, and he is much more than grateful to Professor Schittkowski for his warm hospitality. The author appreciates the help of Professor Andrew 
  R. Conn and an anonymous referee. Their comments have substantially improved the paper.

\end{acknowledgements}

\bibliographystyle{spmpsci}      
\bibliography{snq}   


\end{document}